\documentclass[12pt]{article}
\usepackage{amssymb,amsmath,latexsym}
\usepackage{amsbsy}
\usepackage{amsfonts}
\usepackage{amsopn}
\usepackage{amsthm}
\usepackage{amsmath,verbatim}
\usepackage{color}

\newtheorem{theorem}{Theorem}

\newtheorem{lemma}[theorem]{Lemma}

\newtheorem{defn}{Definition}[section]

\numberwithin{equation}{section}

\def\qed{{\hfill $\Box$ \bigskip}}

\def\Rd{\mathbb{R}^d}

\def\P{{\mathbb P}}

\def\pf{\noindent{\bf Proof.} }
\title{The backbone decomposition for superprocesses with non-local branching}
\author{{\sc A. Murillo-Salas\footnote{ Departamento de Matem\'aticas, Universidad de Guanajuato,
Jalisco s/n, Mineral de Valenciana,
Guanajuato, Gto. C.P. 36240, M\'exico. E-mail: amurillos@ugto.mx}\ \ and \ \ J.L. P\'erez\footnote{ Department of Probability and Statistics, IIMAS-UNAM, 01000 Mexico, D.F. E-mail: garmendia@sigma.iimas.unam.mx
}
}}

\begin{document}
\maketitle
\begin{abstract}
 We provide a path-wise "backbone" decomposition for supercritical superprocesses with non-local branching. Our result complements a related result obtained for supercritical superprocesses without non-local branching in \cite{am}. Our approach relies heavily on the use of so-called  Dynkin-Kuznetsov $\mathbb{N}$-measures.

\vspace{0.5cm}


\end{abstract}
\section{Introduction} 
 In this note we consider  any superprocess $X=\{X_t:t\geq0\}$  on $\mathbb{R}^d$ which is well defined for initial 
configurations 
 $\mu\in\mathcal{M}_C(\Rd)$, the space of finite and compactly supported measures, having  associated a conservative diffusion 
semigroup $\mathcal{P}=\{\mathcal{P}_t:t\geq0\}$ on $\mathbb{R}^d$ and a branching mechanism $\psi$ of the form
\begin{equation}\label{bm1}
\psi(x,f,z)=\psi^{L}(x,z)+\psi^{NL}(x,f),\,\,\,\,x\in\Rd,\,z\geq0, f\in B^+(\Rd),
\end{equation}
where $B^+(\Rd)$ denotes the set of positive measurable functions on $\Rd$, i.e., we consider  superprocesses with non-local branching 
(See \cite{dlg}). The first term corresponds to the branching mechanism related to the local branching of the superprocess $X$, and according to \cite{dlg} it takes 
the following form
\begin{equation}
\psi^{L}(x,z)=\alpha(x) z+\beta(x) z^2+\int_0^{\infty}(e^{-zu}-1+zu)\Pi^{L}(x,du),\,\,\,\,x\in\Rd,\,z\geq0,\notag
\end{equation}
 for bounded measurable functions $\alpha:\mathbb{R}^d\to\mathbb{R}$, $\beta:\mathbb{R}^d\to\mathbb{R}_+$, and $(u\wedge u^2)\Pi^{L}$ is a bounded 
 kernel from  $\mathbb{R}^d$ to $(0,\infty)$ (i.e. the application $x\to\int_{\Rd}(u\wedge u^2)\Pi^{L}(x,du)$ is bounded on $\Rd$). On the other hand, the second term in the right hand side of (\ref{bm1}) is related to non-local 
 branching which takes the form (cf. \cite{dlg}) 
\begin{equation}\label{bm}
\psi^{NL}(x,f)=(f(x)-\zeta(x,f)),\qquad\text{$x\in \Rd$, $f\in B^+(\mathbb{R}^d)$},\notag
\end{equation}
with
\begin{equation}
\zeta(x,f)=\int_{M_0(\Rd)}\left(\gamma(x,\pi)\pi(f)+\int_0^{\infty}(1-e^{-u\pi(f)})\Pi^{NL}(x,\pi,du)\right)G(x,d\pi),\notag
\end{equation}
where $\gamma\in B^+(\Rd\times M_0(\mathbb{R}^d))$ ($M_0(\mathbb{R}^d)$ denotes the set of probability measures on $\mathbb{R}^d$), $u\Pi^{NL}(x,\pi,du)$ is a bounded kernel from 
$\Rd\times M_0(\mathbb{R}^d)$ to $(0,\infty)$ and $G(x,d\pi)$ is a probability kernel from $\Rd$ to $M_0(\mathbb{R}^d)$ with
\begin{equation}
\gamma(x,\pi)+\int_0^{\infty}u\Pi^{NL}(x,\pi,du)\leq 1.\notag
\end{equation}
In fact,  $X$ is a Markovian $\mathcal{M}_C(\Rd)$-valued process whose one-dimensional distributions are characterised by the following result
\begin{lemma} (Lemma 3.3 in \cite{dlg})
For all $f\in bp(\Rd)$, the space of non-negative, bounded measurable functions on $\mathbb{R}^d$,
\begin{equation}
-\log \mathbb{E}_\mu(e^{-\langle f, X_t\rangle})=\int_{\mathbb{R}^d}u_f(x,t)\mu(dx), \qquad\text{$\mu\in\mathcal{M}_C(\mathbb{R}^d)$, $t\geq0.$}\notag
\end{equation}
where $u_f(x,t)$ is the unique non-negative solution to the integral equation
\begin{equation}\label{ed1}
u_f(x,t)=\mathcal{P}_t[f](x)-\int_0^t\mathcal{P}_s[\psi^{L}(\cdot,u_f(\cdot,t-s))+\psi^{NL}(\cdot,u_f(\cdot,t-s))](x).
\end{equation}
\end{lemma}
We call $(X,\P_\mu)$ a $(\mathcal{P},\psi^L,\psi^{NL})$-superprocess started at $\mu\in\mathcal{M}_C(\Rd)$. 

The goal of this note is to give a path-wise backbone decomposition for a $(\mathcal{P},\psi^L,\psi^{NL})$-superprocess, similar
to the work  \cite{am} where the non-local branching is not considered. Loosely speaking, the backbone decomposition is a way
to reconstruct a supercritical superprocess from a branching particle system ({\it called the backbone}) together with some sources
($(\mathcal{P},\psi^L,\psi^{NL})$-superprocesses conditioned to die)
of Poissonian immigration along the paths of the particles in the backbone. Such a decomposition has been done in \cite{EvansOconnel} for 
a quadratic superprocess from the analitic point of view. Since then there has been a lot of interest in finding a path-wise 
backbone decomposition for several different models of superprocesses  due to a variety of applications that have been found 
(e.g. \cite{KMSP,Milos}).

 Very recently, in  \cite{KPR}, the authors provide the backbone decomposition for a quite general spatially dependent supercritical 
superprocess without non-local branching. See \cite{KPR} Section 2 for a summary of some backbone decompositions found in the literature.  
Here, we are interested in the effects that the non-local 
branching has on the backbone decomposition, hence thoroughout this paper we drop out the assumption of having a
spatially dependent branching mechanism. Namely, we consider
\begin{equation*}
\psi^{L}(z)=\alpha z+\beta z^2+\int_0^{\infty}(e^{-z u}-1+z u)\Pi^{L}(du),\,\,z\geq0,
\end{equation*}
with $\alpha\in\mathbb{R}$, $\beta\geq0$, and $\Pi^L$ a measure concentrated in $(0,\infty)$ such that  
$\int_0^{\infty}(u\wedge u^2)\Pi^{L}(du)<\infty$. For the non-local branching we  assume  that the probability kernel $G(x,d\pi) \equiv$ unit mass 
at some $\pi(x,\cdot)\in M_0(\mathbb{R}^d)$. For a measurable function  $f$ we set  $\pi(x,f)\equiv\int_{\Rd}f(y)\pi(x,dy)$. In this case, the non-local branching mechanism is given by
\begin{equation}
\psi^{NL}(x,f)=f(x)-\zeta(\pi(x,f)),\qquad\text{$x\in \mathbb{R}^d$, $f\in B^+(\mathbb{R}^d)$},\notag
\end{equation}
where
\begin{equation}
\zeta(\lambda)=\gamma\lambda+\int_0^{\infty}(1-e^{-\lambda u})\Pi^{NL}(du),\qquad\text{ $\lambda\geq0$},\notag
\end{equation}
where $\gamma\geq0$ and $\int_0^\infty u\Pi^{NL}(du)<\infty$ is such that 
\begin{equation}\label{c1}
\gamma+\int_0^{\infty}u\Pi^{NL}(du)\leq 1.\notag
\end{equation}
Putting all together the above assumptions, we get that the mild equation (\ref{ed1}) satisfied by the semigroup $u_f$ can be written as 
\begin{equation}\label{ed2}
u_f(x,t)=\mathcal{P}_t[f](x)-\int_0^t\mathcal{P}_s[\phi^{L}(u_f(\cdot,t-s))+\phi^{NL}(\cdot,u_f(\cdot,t-s))](x),
\end{equation}
where $\phi^L(z)=\psi^L(z)+z$ for $z\geq0$, and $\phi^{NL}(x,f)=\psi^{NL}(x,f)-f(x)$ for $x\in \Rd$, $f\in B^+(\mathbb{R}^d)$.

The note is organised as follows. Section 2 contains the backbone decomposition given in \cite{am}, when non-local branching is not taken into account. In Section 3 we obtain the superprocess $X$ conditioned to die and characterise the prolific individuals
which are responsible for the infinite growth of the total mass. Finally,  Section 4 provides the backbone decomposition.

\section{The backbone decomposition without non-local branching}

The  so-called  backbone decompositions have been known in the earlier and more analytical setting of semigroup decompositions through the work of \cite{EvansOconnel} and \cite{EP} as well as in the pathwise setting in the work of \cite{SV1, SV2}.  The purpose of this section is to introduce the pathwise backbone decomposition for a supercritical 
superprocess without non-local branching given in \cite{am}, we hope this will make the rest of the paper easier to follow. 


To describe the backbone decomposition in detail, 
consider the process $\{\Lambda^X_t : t \geq 0\}$ which has the following pathwise construction.
First sample from a branching particle diffusion with branching generator
\begin{equation}
F(r) = q\left(\sum_{n\geq 0} p_n r^n - r\right) =  \frac{1}{\lambda^*}\psi(\lambda^*(1-r)), \, r\in[0,1],
\label{F}
\end{equation}
and particle motion which is that of a Markov process with semigruop $\mathcal{P}$.
Note that in the above generator, we have that $q$ is the rate at which individuals reproduce and $\{p_n: n\geq 0\}$ is the offspring distribution. With the particular branching generator given by \eqref{F}, $q = \psi'(\lambda^*)$, $p_0 = p_1 =0$, and for $n\geq 2$,  $p_n : = p_n[0,\infty)$ where for $y\geq 0$, we defined the measure $p_n(\cdot)$ on $\{2,3,4,\ldots\}\times[0,\infty)$ by
\[
p_n({\rm d}y) =  \frac{1}{\lambda^* \psi'(\lambda^*)}\left\{\beta (\lambda^*)^2\delta_0({\rm d}y)\mathbf{1}_{\{n=2\}} + (\lambda^*)^n \frac{y^n}{n!} e^{-\lambda^*y} \Pi({\rm d}y)\right\}.
\]
If we denote the aforesaid branching particle diffusion by  $Z^X = \{Z^X_t: t\geq 0\}$ then we shall also insist that the configuration of particles in space at time zero, $Z_0$, is given by an independent    Poisson random measure with intensity $\lambda^*\mu$.
Next, 
{\it dress} the branches of the spatial tree that describes
the trajectory of $Z^X$ in such a way that a particle at the
space-time position $(\xi, t)\in[0,\infty)^2$ has an independent
$\mathcal{X}$-valued trajectory grafted on to it with rate
\[
2\beta {\rm d}\mathbb{N}_\xi^* + \int_0^\infty y e^{-\lambda^* y}\Pi({\rm d}y){\rm d}\mathbb{P}^{*}_{\xi\delta_y}.
\]
The measures $\{\mathbb{N}_x, x\in \Rd\}$ are the so-called Dynkin-Kuznetsov measures  (see \cite{DK}) , which satisfy
\begin{equation}\label{NmeasureLaplace}
\mathbb{N}_x\left(1-e^{-\langle f,X_t\rangle}\right)=-\log\mathbb{E}_x\left(e^{-\langle f,X_t\rangle}\right),
\end{equation}
for all $f\in bp(\Rd)$ and $t\geq0$.  The measures $\{\mathbb{N}_x, x\in \Rd\}$
play the role of the L\'evy-measure (in the space of measure-valued cadlag paths $\mathcal{X}$) for the infinite divisible measure $\P_{\delta_x}$. The measure $\mathbb{N}_x^*$ denotes the Dynkin-Kuznetsov measure associated to the superprocess conditioned to die.
 Moreover, on the event that an individual in $Z^X$ dies and branches into $n\geq 2$ offspring at spatial position $\xi\in[0,\infty)$, with probability $p_n({\rm d}y)\mathbb{P}^{*}_{y\delta_\xi}$, an additional independent $\mathcal{X}$-valued trajectory is grafted on to the space-time branching point. The quantity $\Lambda^X_t$ is now understood to be the total dressed mass present at time $t$ together with the mass present at time $t$ of an independent copy of $(X,\mathbb{P}^{*}_{\mu})$ issued at time zero.
We denote the law of $(\Lambda^X, Z^X)$ by $\mathbf{P}_{\mu}$.

The backbone decomposition is now summarised by the following theorem lifted from Berestycki et al. \cite{am}.

\begin{theorem}\label{main-1}
For any $\mu\in\mathcal{M}_F(\mathbb{R}^d)$, the process $(\Lambda^X,  \mathbf{P}_\mu)$ is Markovian and has the same law as  $(X,  \mathbb{P}_\mu )$. Moreover, for each $t\geq 0$, the law of $Z^X_t$ given $\Lambda^X_t$ is that of a Poisson random measure with intensity measure $\lambda^*\Lambda^X_t$.
\end{theorem}

\section{The conditioned superprocess and prolific individuals}
\subsection{The conditioned superprocess}
We note that the total mass process, $\|X\|:=\{\|X_t\|\equiv\langle 1,X_t\rangle,t\geq0\}$, is a continuous state branching process with branching mechanism $\bar{\psi}$ given by
\begin{equation}\label{CSBP}
\bar{\psi}(\lambda):=(\alpha+1)\lambda+\beta\lambda^2+\int_0^{\infty}(e^{-\lambda u}-1+\lambda u)\Pi^{L}(du)-\gamma\lambda-
\int_0^{\infty}(1-e^{-\lambda u})\Pi^{NL}(du). 
\end{equation}
In order to avoid explosion of the total mass in finite time we assume that $\int_{0+}1/|\bar{\psi}(\xi)|d\xi=\infty$ (see \cite{Grey(1974)}). We will assume that the branching 
mechanism (\ref{CSBP}) is supercritical in the sense that  $0<-\bar{\psi}^\prime(0+)<\infty$, thus the mean-total mass grows exponentially at rate 
$-\bar{\psi}^\prime(0+)$. Under the above assumptions, and recalling the fact that $\bar{\psi}$
is strictly convex (\cite{Grey(1974)}), there exists  a unique $\lambda^*>0$ such that $\bar{\psi}(\lambda^*)=0$. Moreover, for all $\mu\in\mathcal{M}_C(\Rd)$,
\[\P_\mu(\lim_{t\uparrow\infty}\|X_t\|=0)=e^{-\|\mu\|\lambda^*}.\]
We also assume  the condition
\[\int^\infty \frac{1}{\overline{\psi}(\xi)}d\xi<\infty,\]
which ensures that the event $\{\lim_{t\uparrow\infty}\|X_t\|=0\}$ agrees with the event of extinction $\{\zeta<\infty\}$, with $\zeta=\inf\{t>0:\|X_t\|\}$ (e.g. see \cite{Grey(1974)} and \cite{bing}).

We can express the probability of survival in terms of the so-called Dynkin-Kuznetsov measures $\{\mathbb{N}_x, x\in \Rd\}$  as follows.
 Set 
$\mathcal{E}:=\{\lim_{t\uparrow \infty}\|X_t\|=0\}$ then we have
\begin{equation}
\mathbb{P}_\mu(\mathcal{E})=e^{-\mathbb{N}_{\mu}(\mathbf{1}_S)}=e^{-\int_{\Rd}\mathbb{N}_x(\mathbf{1}_S)\mu(dx)}=e^{-\lambda^*\|\mu\|},\notag
\end{equation}
where $S$ denotes the event of survival.  Using the probability of extinction for the superprocess $X$ we can now prove the following 
\begin{lemma}\label{conditionedsuperprocess}
For each $\mu\in\mathcal{M}_C(\mathbb{R}^d)$, define the law of $X$ with initial configuration $\mu$ conditioned on becoming extinct by $\mathbb{P}_\mu^*$. Specifically, for all events $A$, measurable in the natural sigma algebra of $X$,
\begin{equation}
\mathbb{P}^*_{\mu}(A)=\mathbb{P}_{\mu}(A|\lim_{t\uparrow\infty}\|X_t\|=0).\notag
\end{equation}
Then, for all bounded $f$
\begin{equation}
-\log\mathbb{E}_{\mu}^*(e^{-\langle f,X_t\rangle})=\int_{D}u_f^*(x,t)\mu(dx),\notag
\end{equation}
with
\begin{equation}
u_f^*(x,t)=u_{f+\lambda^*}(x,t)-\lambda^*,\notag
\end{equation}
where $u_f^*(x,t)$ is the unique solution of the integral equation
\begin{equation}\label{ed3}
u^*_f(x,t)=\mathcal{P}_t(f)(x)-\int_0^t\mathcal{P}_s[\phi^{L,*}(u^*_f(\cdot,t-s))+\phi^{NL,*}(\cdot,u^*_f(\cdot,t-s))](x),
\end{equation}
where $\phi^{L,*}(\lambda)=\phi^L(\lambda+\lambda^*)$ for $\lambda\geq-\lambda^*$ and $\phi^{NL,*}(x,f)=\phi^{NL}(x,f+\lambda^*)$ for any positive measurable function $f$ such that $f+\lambda^*\in B^+(\mathbb{R}^d)$. That is to say $(X, \mathbb{P}_{\mu}^*)$ is a 
$(\mathcal{P},\phi^{L,*},\phi^{NL,*})$-superprocess.
\end{lemma}
\pf 
Set $\mathcal{E}=\{\lim_{t\uparrow\infty}\|X_t\|=0\}$, then
\begin{align}
\mathbb{E}_{\mu}^*(e^{-\langle f,X_t\rangle})&=\mathbb{E}_{\mu}(e^{-\langle f,X_t\rangle}|\mathcal{E})\notag\\
&=e^{\lambda^*\|\mu\|}\mathbb{E}_{\mu}(e^{-\langle f,X_t\rangle}1_{\mathcal{E}})\notag\\
&=e^{\lambda^*\|\mu\|}\mathbb{E}_{\mu}(e^{-\langle f,X_t\rangle}\mathbb{E}_{X_t}(1_{\mathcal{E}}))\notag\\
&=e^{\lambda^*\|\mu\|}\mathbb{E}_{\mu}(e^{-\langle f+\lambda^*,X_t\rangle})\notag\\
&=e^{-\langle u_{f+\lambda^*}(\cdot,t)-\lambda^*,\mu\rangle}.\notag
\end{align}
Now,  using (\ref{ed2}) it is easy to check that $u_f^*(x,t)=u_{f+\lambda^*}(x,t)-\lambda^*$ is a solution to
\begin{equation}
u^*_f(x,t)=\mathcal{P}_t(f)(x)-\int_0^t\mathcal{P}_s[\phi^{L,*}(u^*_f(\cdot,t-s))+\phi^{NL,*}(\cdot,u^*_f(\cdot,t-s))](x),\notag
\end{equation}
where $\phi^{L,*}(\lambda)=\phi^{L}(\lambda+\lambda^*)$ and $\phi^{NL,*}(\cdot,f)=\phi^{NL}(\cdot,f+\lambda^*)$.
\qed
\subsection{Prolific individuals}
We will now identify the branching mechanism of the backbone for the superprocess $X$, i.e., we will give the generator of the continuous-time Galton Watson process related to the genealogies responsible for the infinite growth of the process, in the form
\begin{equation}
F(x,s)=q\left(\sum_{n\geq0}p^{L}_n+\sum_{n\geq0}p^{NL}_n\right)(s^n-s),\notag
\end{equation}
where $q>0$ is the common rate of splitting and $\{p_n^L:n\geq0\}$ is the offspring distribution related to local branching, i.e. $p_n^{L}$ is the probability of having $n$ offspring at the position in which the parent dies. 
Respectively, $p_n^{NL}$ is the probability of having $n$ offspring displaced from the position $x$ of the death of the parent  according to a random variable $\Theta$ such that $\Theta+x$ has distribution $\pi(x,\cdot)$.
\par Moreover the branching rate is given by $q\equiv(\phi^{L})^\prime(\lambda^*)$ (we leave it to the reader to verify that $q>0$), $p^{L}_0=p^L_1=0$, $n\geq2$,
\begin{equation}\label{pl1}
p^{L}_n=\frac{1}{\lambda^*q}\left\{\beta (\lambda^*)^2\mathbf{1}_{\{n=2\}}+\int_{(0,\infty)}\frac{(y\lambda^*)^n}{n!}e^{-\lambda^*y}\Pi^L(dy)\right\}.
\end{equation}
For the non-local offspring distribution we have that $p^{NL}_0=0$ and for $n\geq1$,
\begin{equation}\label{pl2}
p^{NL}_n=\frac{1}{\lambda^*q}\left\{\lambda^*\gamma\mathbf{1}_{\{n=1\}}+\int_{(0,\infty)}\frac{(y\lambda^*)^n}{n!}e^{-\lambda^*y}\Pi^{NL}(dy)\right\}.
\end{equation}
 We leave to the reader to verify that effectively $\sum_{n\geq1}(p_n^{L}+p^{NL}_n)=1$.  
On the other hand, to describe the law related to the discontinuous immigration along the backbone, once again we will deal with the  local and non-local immigration separately. For the local immigration we have that
\begin{equation}\label{il1}
\eta^{L}_n(dy)=\frac{1}{p^L_n\lambda^*q}\left\{\beta (\lambda^*)^2\delta_0(dy)\mathbf{1}_{\{n=2\}}+\frac{(y\lambda^*)^n}{n!}e^{-\lambda^*y}\Pi^{L}(dy)\right\},
\end{equation}
whereas for the non-local type of immigration we have
\begin{equation}\label{il2}
\eta^{NL}_n(dy)=\frac{1}{p_n^{NL}\lambda^*q}\left\{\lambda^*\gamma\delta_0{(dy)}\mathbf{1}_{\{n=1\}}+\frac{(y\lambda^*)^n}{n!}e^{-\lambda^*y}\Pi^{NL}(dy)\right\}.
\end{equation}

\section{Backbone decomposition}
\subsection{A branching particle system with four types of immigration}
Let $\mathcal{M}_a(\Rd)$ be the space of finite atomic measures on $\Rd$. Now suppose that $\xi=\{\xi_t:t\geq0\}$ is the stochastic process whose 
semi-group is given by $\mathcal{P}$. We shall use the expectation operators $\{E_x:x\in \Rd\}$ defined by $E_x(f(\xi_t))=\mathcal{P}_t[f](x)$.  
Let $Z=\{Z_t,t\geq0\}$ be a $\mathcal{M}_a(\Rd)$-valued process in which individuals, from the moment of birth, live for an independent and exponentially distributed 
time with parameter $q$ during which they execute a $\mathcal{P}$-diffusion issued from their position of birth and at death they give birth at the same 
position to an independent number of offspring locally with probabilities $\{p^L_n:n\geq2\}$, and non-locally with probabilities 
$\{p^{NL}_n(\cdot):n\geq1\}$.  Hence, $Z$  is a non-local branching particle system such that 
\begin{equation}
-\log E_{x}(e^{-\langle f,Z_t\rangle})=v_f(x,t),\notag
\end{equation}
where the semigroup $v_f$ satisfies the following integral equation
\begin{align}
e^{-v_f(x,t)}&=e^{-qt}\mathcal{P}_t\left[e^{-f}\right](x)\notag\\
&+\int_0^tds\,qe^{-qs}\mathcal{P}_s\left[\sum_{n=0}^{\infty}p_n^{L}e^{-nv_f(\cdot,t-s)}+\int_{\mathcal{M}_F(\mathbb{R}^d)}\sum_{n=0}^{\infty}e^{-\langle v_f(\cdot,t-s),\nu\rangle}p_n^{NL}(l\pi)^{*n}(d\nu)\notag
\right](x),
\end{align}
where $l\pi(d\nu)$ denotes the image of $\pi$ under the map $y\to\delta_y$ from $\mathbb{R}^d$ to $\mathcal{M}_F(\mathbb{R}^d)$ (the space of finite measures on $\mathbb{R}^d$) and $(l\pi)^{*n}$ denotes the $n$-fold convolution of $l\pi$.  
Thus, a parent particle at the position $x\in\Rd$ when branches it  gives birth to a random number of offspring in the following fashion:  
it produces  $n$ new individuals, which are initially located at $x$, 
with probability $p^L_N$; and produces $m$ new individuals, which choose their locations in $\mathbb{R}^d$ independently of each other according to the (non-random)
distribution $\pi(x,\cdot)$, with probability  $p_m^{NL}$.
 We shall  refer to $Z$ as the backbone with initial configuration denoted by
 $\nu\in\mathcal{M}_a(\Rd)$.  We will use the Ulam-Harris notation, i.e.,   that the individuals in $Z$  are uniquely identifiable amongst   $\mathcal{T}$,  
 the set labels of individuals realised in $Z$. For each individual $u\in \mathcal{T}$ we shall write $\tau_u$ and $\sigma_u$ for its birth and death times respectively, 
 $\{z_u(r): r\in[\tau_u, \sigma_u]\}$ for its spatial trajectory and $N_u$ for the number of offspring it has at time $\sigma_u$.

 With these elements at hands we are able to express the backbone decomposition of the superprocess $X$.  We are interested in immigrating 
$(\mathcal{P},\phi^{L,*},\psi^{NL,*})$-superprocesses along the backbone $Z$ in a way that the rate of 
immigration is related to the subordinator  (i.e. a L\'evy process with a.s. increasing paths), whose 
Laplace exponent is given by
\begin{align}\label{els}
\Phi(\lambda)&=( \phi^L)^\prime(\lambda+\lambda^*)-( \phi^L)^\prime(\lambda^*)\notag\\
&= 2\beta\lambda+\int_{(0,\infty)}(1-e^{-\lambda y})ye^{-\lambda^*y}\Pi^L(dy),
\end{align}
together with some additional immigration at the splitting times of $Z$.
\begin{defn}
For $\nu\in\mathcal{M}_a(\Rd)$ and $\mu\in\mathcal{M}_C(\Rd)$ let $Z$ be a $(\mathcal{P},F)$-branching diffusion with initial configuration 
$\nu$ and $\bar{X}$ and independent copy of $X$ under $\mathbb{P}_{\mu}^*$. Then we define the measure-valued stochastic process 
$\Delta=\{\Delta_t:t\geq0\}$ on $\Rd$ by
\begin{equation}
\Delta=\bar{X}+I^{\mathbb{N}^*}+I^{\mathbb{P}^*}+I^{\eta,L}+I^{\eta,NL},\notag
\end{equation}
where the processes $I^{\mathbb{N}^*}$, $I^{\mathbb{P}^*}$, $I^{\eta,L}$, and $I^{\eta,NL}$ are independent of $\bar{X}$ and, conditionally on $Z$, 
are independent of each other. More precisely, these processes are described as follows:
\begin{itemize}
\item[1]\textbf{Continuous immigration:} The process $I^{\mathbb{N}^*}$ is measure-valued on $\Rd$ such that
\begin{equation}
I_t^{\mathbb{N}^*}=\sum_{u\in\mathcal{T}}\sum_{t\wedge\tau_u<r\leq t\wedge\sigma_u}X_{t-r}^{(1,u,r)}\notag
\end{equation}
where, given $Z$, independently for each $u\in\mathcal{T}$ such that $\tau_u<t$, the processes $X_{\cdot}^{(1,u,r)}$ are countable in number and correspond to $\chi$-valued,
Poissonian immigration along the space-time trajectory $\{(z_u(r),r):r\in(\tau_u,t\wedge\sigma_u]\}$ with rate $2\beta dr\times d\mathbb{N}^*_{z_u(r)}$.
\item[2]\textbf{Discontinuous immigration:} The process $I^{\mathbb{P}^*}$ is measure-valued on $\Rd$ such that
\begin{equation}
I_t^{\mathbb{P}^*}=\sum_{u\in\mathcal{T}}\sum_{t\wedge\tau_u<r\leq t\wedge\sigma_u}X_{t-r}^{(2,u,r)}\notag
\end{equation}
where, given $Z$, independently for each $u\in\mathcal{T}$ such that $\tau_u<t$, the processes $X_{\cdot}^{(2,u,r)}$ are countable in number and 
correspond to $\chi$-valued, Poissonian immigration along the space-time trajectory $\{(z_u(r),r):r\in(\tau_u,t\wedge\sigma_u]\}$ with rate 
\[dr\times\int_{y\in(0,\infty)}ye^{-\lambda^*y}\Pi^L(dy)\times d\mathbb{P}^*_{y\delta_{z_u(r)}}.\]
\item[3]\textbf{Local Branch point biased immigration:} The process $I^{\eta,L}$ is measure-valued on $\Rd$ such that
\begin{equation}
I_t^{\eta,L}=\sum_{u\in\mathcal{T}}1_{\{\sigma_u\leq t\}}X_{t-\sigma_u}^{(3,u)}\notag
\end{equation}
where, given $Z$, independently for each $u\in\mathcal{T}$ such that $\sigma_u<t$, the processes $X_{\cdot}^{(3,u)}$ is an independent copy of $X$ issued at time $\sigma_u$ with law $\mathbb{P}_{Y_u\delta_{z_u(\sigma_u)}}$ where $Y_u$ is an independent random variable with distribution $\eta^L_{N_u}(dy)$.
\item[4]\textbf{Non-local Branch point biased immigration:} The process $I^{\eta,NL}$ is measure-valued on $\Rd$ such that
\begin{equation}
I_t^{\eta,NL}=\sum_{u\in\mathcal{T}}1_{\{\sigma_u\leq t\}}X_{t-\sigma_u}^{(3,u)}\notag
\end{equation}
where, given $Z$, independently for each $u\in\mathcal{T}$ such that $\sigma_u<t$, the processes $X_{\cdot}^{(3,u)}$ is an independent copy of $X$ issued at time $\sigma_u$ with law $\mathbb{P}_{Y_u\pi(z_u(\sigma_u),\cdot)}$ where $Y_u$ is an independent random variable with distribution $\eta^{NL}_{N_u}(dy)$. 
\end{itemize}
Moreover, we denote the law of $\Delta$ by $\mathbf{P}_{\mu\times\nu}$.
\end{defn}
We will now state our first theorem
\begin{theorem}\label{keyTh}
For every $\mu\in \mathcal{M}_C(\Rd)$, $\nu\in\mathcal{M}_a(\Rd)$ and $f,h\in bp(\Rd)$ we have that
\begin{equation}\label{f2}
\mathbf{E}_{\mu\times\nu}(e^{-\langle f,\Delta_t\rangle-\langle f,Z_t\rangle})=e^{-\langle u_f^*(\cdot,t),\mu\rangle-\langle v_{f,h}(\cdot,t),\nu\rangle},
\end{equation}
where $\exp\{-v_{f,h}(x,t)\}$ is the unique $[0,1]$-valued solution to the integral equation
\begin{align}\label{ed4}
e^{-v_{f,h}(x,t)}&=\mathcal{P}_t\left[e^{-h}\right](x)+\frac{1}{\lambda^*}\int_0^t\mathcal{P}_s[\phi^{L,*}(-\lambda^*e^{-v_{f,h}(\cdot,t-s)}+u_f^*(\cdot,t-s))\notag\\
&-\phi^{L,*}(u_f^*(\cdot,t-s))+\phi^{NL,*}(-\lambda^*e^{-v_{f,h}(\cdot,t-s)}+u_f^*(\cdot,t-s))\notag\\
&-\phi^{NL,*}(u_f^*(\cdot,t-s))](x),
\end{align}
for all $x\in\Rd$ and $t\geq0$.
\end{theorem}
In order to prove the Theorem \ref{keyTh} we will need to prove first some preliminary results.
\begin{lemma}
For all $f\in bp(\Rd)$, $\mu\in \mathcal{M}_C(\Rd)$, $\nu\in\mathcal{M}_a(\Rd)$ and $t\geq0$, we have
\begin{equation}
\mathbf{E}_{\mu\times\nu}(e^{-\langle f, I^{\mathbb{N}^*}_t+I^{\mathbb{P}^*}_t\rangle}|\{Z_s:s\leq t\})=\exp\left\{-\int_0^t\langle\Phi(u_f^*(\cdot,t-s)),Z_s\rangle ds\right\},\notag
\end{equation}
where $\Phi$ is given by (\ref{els}).
\end{lemma}
\pf We write
\begin{equation}
\langle f, I^{\mathbb{N}^*}_t+I^{\mathbb{P}^*}_t\rangle=\sum_{u\in\mathcal{T}}\sum_{t\wedge\tau_u<r\leq t\wedge\sigma_u}\langle f,X_{t-r}^{(1,u,r)}\rangle+\sum_{u\in\mathcal{T}}\sum_{t\wedge\tau_u<r\leq t\wedge\sigma_u}\langle f,X_{t-r}^{(2,u,r)}\rangle.\notag
\end{equation}
Hence conditioning on $Z$, appealing to the independence of the immigration processes together with Campbell's formula (see e.g. Theorem 2.7 in \cite{K})
\begin{align}
\mathbf{E}_{\mu\times\nu}&(e^{-\langle f, I^{\mathbb{N}^*}_t\rangle}|\{Z_s:s\leq t\})=\exp\left\{-\sum_{u\in\mathcal{T}}2\int_{t\wedge\tau_u}^{t\wedge\sigma_u}\beta\cdot\mathbb{N}^*_{z_u(r)}(1-e^{-\langle f,X_{t-r}\rangle})dr\right\}\notag.
\end{align}
Now using that $\mathbb{N}^*_{z_u(r)}(1-e^{-\langle f,X_{t-r}\rangle})=u_f^*(z_u(r),t-r)$ (see (\ref{NmeasureLaplace})), we have
\begin{align}\label{i1}
\mathbf{E}_{\mu\times\nu}&(e^{-\langle f, I^{\mathbb{N}^*}_t\rangle}|\{Z_s:s\leq t\})
=\exp\left\{-\sum_{u\in\mathcal{T}}2\beta\int_{t\wedge\tau_u}^{t\wedge\sigma_u}u_f^*(z_u(r),t-r)dr\right\}.
\end{align}
On the other hand
\begin{align}\label{i2}
\mathbf{E}_{\mu\times\nu}&(e^{-\langle f, I^{\mathbb{P}^*}_t\rangle}|\{Z_s:s\leq t\})\notag\\
&=\exp\left\{-\sum_{u\in\mathcal{T}}\int_{t\wedge\tau_u}^{t\wedge\sigma_u}dr\int_0^\infty ye^{-\lambda^*} \Pi^L(dy)\mathbb{P}^*_{y\delta_{z_u(r)}}(1-e^{-\langle f,X_{t-r}\rangle})\right\}\notag\\
&=\exp\left\{-\sum_{u\in\mathcal{T}}\int_{t\wedge\tau_u}^{t\wedge\sigma_u}dr\int_0^\infty ye^{-\lambda^*} \Pi^L(dy)(1-e^{-yu_f^*(z_u(r),t-r)})\right\}.
\end{align}
Then, using (\ref{els}), (\ref{i1}) and (\ref{i2}) we get that
\begin{align}
\mathbf{E}_{\mu\times\nu}(e^{-\langle f, I^{\mathbb{N}^*}_t+I^{\mathbb{P}^*}_t\rangle}&|\{Z_s:s\leq t\})
=\exp\bigg\{-\sum_{u\in\mathcal{T}}\int_{t\wedge\tau_u}^{t\wedge\sigma_u}dr\bigg(2\beta u_f^*(z_u(r),t-r)\notag\\
&+\int_0^\infty ye^{-\lambda^*}\Pi^L(dy)(1-e^{-yu_f^*(z_u(r),t-r)})\bigg)\bigg\}\notag\\
&=\exp\left\{-\sum_{u\in\mathcal{T}}\int_{t\wedge\tau_u}^{t\wedge\sigma_u}dr(\Phi(u_f^*(z_u(r),t-r)))\right\}\notag\\
&=\exp\left\{-\int_0^t\langle \Phi(u_f^*(\cdot,t-r)),Z_r\rangle dr\right\}.\notag
\end{align}
\qed 

In the next lemma we shall use the notation \[\pi(\cdot,f(\circ,t))\equiv\int_{\Rd}f(y,t)\pi(\cdot,dy),\]
for a measurable function $f$.
\begin{lemma}
Suppose that $f,h\in bp(\Rd)$ and $g_s(x)$ is jointly measurable in $(s,x)$ and bounded on finite time horizonts of $s$. Then for 
all $x\in \Rd$ and $t\geq0$,
\begin{equation}
\mathbf{E}_{\mu\times\nu}\left(\exp\left\{-\int_0^t\langle g_{t-s},Z_s\rangle ds-\langle f, I_t^{\eta,L}\rangle-\langle f, I_t^{\eta,NL}\rangle-\langle h, Z_t\rangle\right\}\right)=e^{-\langle w(\cdot,t),\nu\rangle},\notag
\end{equation}
where $\exp\{-w(x,t)\}$ is the unique $[0,1]$-valued solution to the integral equation

\begin{align}\label{ed5}
e^{-w(x,t)}&=\mathcal{P}_t[e^{-h}](x)+\frac{1}{\lambda^*}\int_0^t\mathcal{P}_s[H_{t-s}(\cdot,-\lambda^*e^{-w(\cdot,t-s)})-\lambda^*e^{-w(\cdot,t-s)}g_{t-s}(\cdot)](x)ds.
\end{align}
where

\begin{align}
H_{t-s}(\cdot,-\lambda^*e^{-w(\cdot,t-s)})&=-\lambda^*qe^{-w(\cdot,t-s)}+\beta(\lambda^*)^2e^{-2w(\cdot,t-s)}+\gamma\lambda^*\pi(\cdot,e^{-w(\circ,t-s)})\notag\\
&+\int_{(0,\infty)}(e^{\lambda^* ye^{-w(\cdot,t-s)}}-1-\lambda^* ye^{-w(\cdot,t-s)})e^{-(\lambda^*+u_f^*(\cdot,t-s))y}\Pi^L(y)\notag\\
&+\int_{(0,\infty)}(e^{\lambda^*\pi(\cdot,e^{-w(\circ,t-s)}) y}-1)e^{-\pi(\cdot,\lambda^*+u_f^*(\circ,t-s))}\Pi^{NL}(dy).\notag
\end{align}

\end{lemma}
\pf 
Following the proof of Theorem 2.2 in \cite{EvansOconnel} it is enough to prove the result for $g$ being time-independent. Recall that 
$\xi=\{\xi_t:t\geq0\}$ is the stochastic process whose semi-group is given by $\mathcal{P}$. Let us define a new semigroup
\begin{equation}
\mathcal{P}^{g}_t[f](x)=E_x\left(e^{-\int_0^tg(\xi_s)}f(\xi_s)\right),\notag
\end{equation}
for $f,g\in bp(\Rd)$. Standard Feynman-Kac manipulations (cf. Lemma 2.3 in \cite{EvansOconnel}) give us that
\begin{equation}\label{fk}
\mathcal{P}^{g}_t[f](x)=\mathcal{P}_t[f](x)-\int_0^tds\mathcal{P}_s[g(\cdot)\mathcal{P}^{g}_{t-s}[f](\cdot)](x).
\end{equation}
Conditioning on the first branching time, and recalling that the branching occurs at rate $q=({\phi^L})^\prime(\lambda^*)$ we get that

\begin{align}\label{i3}
e^{-w(x,t)}=e^{-qt}\mathcal{P}^{g}_t[e^{-h}](x)+&\int_0^tds\cdot qe^{-qs}\mathcal{P}^{g}_s\bigg[\sum_{n\geq2}p^{L}_ne^{- nw(\cdot,t-s)}\int_{(0,\infty)}\eta^{L}_n(dy)e^{-yu_f^*(\cdot,t-s)}\notag\\
&+\int_{\mathcal{M}_C(\mathbb{R}^d)}\sum_{n\geq1}e^{-\langle w(\cdot,t-s),\nu\rangle}p_n^{NL}(l\pi)^{*n}(d\nu)\int_0^{\infty}\eta^{NL}_n(dy)e^{-y\pi(\cdot, u^*_f(\circ,t-s))}\bigg](x).
\end{align}

Using (\ref{pl2}), (\ref{il2}) and performing similar computations  to the ones in \cite{dlg} (cf. Section 3) we have that

\begin{align}\label{cg1}
\sum_{n\geq1}&p_n^{NL}(l\pi)^{*n}(d\nu)\int_0^{\infty}\eta^{NL}_n(dy)e^{-y\pi(\cdot, u^*_f(\circ,t-s))}\notag\\
&=\sum_{n\geq1}\int_0^{\infty}\pi(\cdot,e^{-w(\circ,t-s)})^n\frac{1}{\lambda^*q}\left\{\gamma\lambda^*\delta_0(dy)1_{\{n=1\}}+\frac{(y\lambda^*)^n}{n!}e^{-\lambda^*y}\Pi^{NL}(dy)\right\}e^{-y\pi(\cdot, u_f^*(\circ,t-s))}\notag\\
&=\frac{1}{\lambda^*q}\left\{\int_0^{\infty}(e^{\lambda^*y \pi(\cdot,e^{-w(\circ,t-s)})}-1)e^{-\pi(\cdot,\lambda^*+u_f^*(\circ,t-s))}\Pi^{NL}(dy)+\gamma\lambda^*\pi(\cdot,e^{-w(\circ,t-s)})\right\}.
\end{align}

Now for the local branching term we obtain, by proceeding as in the proof of Lemma 4 in \cite{am} and using (\ref{pl1}) and (\ref{il1}), the following

\begin{align}\label{cg2}
\sum_{n\geq2}&p^{L}_ne^{- nw(\cdot,t-s)}\int_{(0,\infty)}\eta^{L}_n(dy)e^{-yu_f^*(\cdot,t-s)}\notag\\
&=\frac{1}{\lambda^*q}\left\{\beta(\lambda^*)^2e^{-2w(\cdot,t-s)}+\int_0^{\infty}(e^{\lambda^*ye^{-w(\cdot,t-s)}}-1-\lambda^*ye^{-w(\cdot,t-s)})e^{-y(\lambda^*+u_f^*(\cdot,t-s))}\Pi^L(dy)\right\}.
\end{align}

Using (\ref{cg1}) and (\ref{cg2}) in (\ref{i3}) we have that
\begin{align}\label{fk3}
&e^{-w(x,t)}\notag\\
&=e^{-qt}\mathcal{P}^{g}_t[e^{-h}](x)+\int_0^tds\cdot e^{-qs}\mathcal{P}^{g}_s\left[\frac{1}{\lambda^*}(H_{t-s}(\cdot,-\lambda^*e^{-w(\cdot,t-s)})+\lambda^*qe^{-w(\cdot,t-s)})\right](x)\notag\\
&=\mathcal{P}^{g}_t[e^{-h}](x)+\int_0^tds\mathcal{P}^{g}_s\left[\frac{1}{\lambda^*}H_{t-s}(\cdot,-\lambda^*e^{-w(\cdot,t-s)})\right](x).
\end{align}
where the second inequality follows from a standard technique found for example in Lemma 4.1.1 of \cite{dy}. Now making the same computations as in \cite{am} we obtain that
\begin{align}\label{fk2}
\int_0^tds&\mathcal{P}_s\left[g(\cdot)\mathcal{P}_{t-s}^g[e^{-h}](\cdot)\right](x)\notag\\
&+\frac{1}{\lambda^*}\int_0^tds\int_0^sdr\mathcal{P}_r\left[g(\cdot)\mathcal{P}_{s-r}^g[H_{t-s}(\cdot,-\lambda^*e^{-\omega(\cdot,t-s)})]\right](x)\notag\\
&=\int_0^tds\mathcal{P}_s\left[g(\cdot)\mathcal{P}_{t-s}^g[e^{-h}](\cdot)\right](x)\notag\\
&+\frac{1}{\lambda^*}\int_0^tdr\mathcal{P}_r\left[g(\cdot)\int_r^tds\mathcal{P}_{s-r}^g[H_{t-s}(\cdot,-\lambda^*e^{-\omega(\cdot,t-s)})](\cdot)\right](x)\notag\\
&=\int_0^tds\mathcal{P}_s\left[g(\cdot)\mathcal{P}_{t-s}^g[e^{-h}](\cdot)\right](x)\notag
\end{align}
\begin{align}
&+\frac{1}{\lambda^*}\int_0^tdr\mathcal{P}_r\left[g(\cdot)\int_0^{t-r}d\theta\mathcal{P}_{\theta}^g[H_{t-\theta-r}(\cdot,-\lambda^*e^{-\omega(\cdot,t-s)})](\cdot)\right](x)\notag\\
&=\int_0^tdr\mathcal{P}_r\bigg[g(\cdot)\bigg\{\mathcal{P}_{t-r}^g[e^{-h}](\cdot)\notag\\
&+\frac{1}{\lambda^*}\int_0^{t-r}d\theta\mathcal{P}_{\theta}^g[H_{t-r-\theta}(\cdot,-\lambda^*e^{-\omega(\cdot,t-s)})](\cdot)\bigg\}\bigg](x)\notag\\
&=\int_0^tds\mathcal{P}_s\left[g(\cdot)e^{-\omega(\cdot,t-s)}\right](x).
\end{align}
Next, we use (\ref{fk}) and (\ref{fk2}) in (\ref{fk3}) to obtain that
\begin{align}
e^{-w(x,t)}&=\mathcal{P}_t[e^{-h}](x)-\int_0^tds\mathcal{P}_s[g(\cdot)\mathcal{P}^g_{t-s}[e^{-h}](\cdot)](x)\notag\\
&+\frac{1}{\lambda^*}\int_0^tds\bigg\{\mathcal{P}_s[H_{t-s}(\cdot,-\lambda^*e^{-w(\cdot,t-s)})](x)\notag\\
&-\int_0^sdr\mathcal{P}_r[g(\cdot)\mathcal{P}_{s-r}^g[H_{t-s}(\cdot,-\lambda^*e^{-w(\cdot,t-s)})]](x)\bigg\}\notag\\
&=\mathcal{P}_t[e^{-h}](x)+\frac{1}{\lambda^*}\int_0^tds\mathcal{P}_s\left[H_{t-s}(\cdot,-\lambda^*e^{-w(\cdot,t-s)})-\lambda^*g(\cdot)e^{-w(\cdot,t-s)}\right](x).\notag
\end{align}
\par The proof is complete as soon as we can establish uniqueness to (\ref{ed5}). The proof is guided by the same arguments as in the proof of Lemma 4 in \cite{am}, i.e., it suffices to check that for each fixed $T>0$, there exists $K>0$ such that
\begin{equation}
\sup_{s\leq T}\sup_{y\in\mathbb{R}^d}|H_s(y,-u(y))-H_s(y,-v(y))|\leq K\sup_{y\in\mathbb{R}^d}|u(y)-v(y)|,\notag
\end{equation}
where $u$ and $v$ are any two measurable mappings from $\mathbb{R}^d$ to $[0,\lambda^*]$, then Lemma 2.1 in \cite{EvansOconnel} gives the result.
To this end we define for $\lambda\geq-\lambda^*$ and $u\geq0$,
\begin{align}
\chi_u^1(\lambda)=\lambda q+\beta(\lambda)^2+\int_{(0,\infty)}(e^{-\lambda y}-1+\lambda y)e^{-(\lambda^*+u)y}\Pi^L(y),\notag
\end{align}
and for any positive measurable function such that $f+\lambda^*\in B(\mathbb{R}^d)$, and $v\geq0$
\begin{align}
\chi_u^2(\lambda)=\gamma\lambda+\int_{(0,\infty)}(e^{-\lambda z}-1)e^{-(\lambda^*+u) z}\Pi^{NL}(dz).\notag
\end{align}
Therefore by definition we have that $H_s(y,-v(y))=\chi^1_{u_f^*(y,t-s)}(-v(y))+\chi^2_{\pi(y,u_f^*(\circ,t-s))}(-\pi(y,v(\circ)))$, for any measurable mapping $v$ from $\mathbb{R}^d$ to $[0,\lambda^*]$.
\par With the help of Lemma 5 in \cite{am} and the fact that $\pi(x,\cdot)$ is a probability measure for every $x\in\mathbb{R}^d$, we can see that for fixed $T>0$,
\begin{align}
\sup_{s\leq T}\sup_{y\in\mathbb{R}^d}|H_s(y,-u(y))&-H_s(y,-v(y))|\leq\sup_{s\leq T}\sup_{y\in\mathbb{R}^d}|\chi^1_{u_f^*(y,t-s)}(-u(y))-\chi^1_{u_f^*(y,t-s)}(-v(y))|\notag\\
&+\sup_{s\leq T}\sup_{y\in\mathbb{R}^d}|\chi^2_{\pi(y,u_f^*(\circ,t-s))}(-\pi(y,-u(\circ)))-\chi^2_{\pi(y,u_f^*(\circ,t-s))}(-\pi(y,-v(\circ)))|\notag\\
&\leq\sup_{0\leq u^*\leq \bar{u}_T}\sup_{y\in\mathbb{R}^d}|\chi^1_{u^*}(-u(y))-\chi^1_{u^*}(-v(y))|\notag\\
&+\sup_{0\leq u^*\leq \bar{u}_T}\sup_{y\in\mathbb{R}^d}|\chi^2_{u^*}(-\pi(y,-u(\circ)))-\chi^2_{u^*}(-\pi(y,-v(\circ)))|\notag\\
&\leq K\sup_{y\in\mathbb{R}^d}|u(-y)-v(-y)|,\notag
\end{align}
where $u$ and $v$ are any two measurable mappings from $\mathbb{R}^d$ to $[0,\lambda^*]$,
\begin{equation}\label{cota}
K=\sup_{0\leq u^*\leq \bar{u}_T}\sup_{y\in\mathbb{R}^d}(|(\chi^1_{u^*})'(-\lambda)|+|(\chi^2_{u^*})'(-\lambda)|)<\infty,
\end{equation}
(observe that using Lemma 5 in \cite{am} we have that (\ref{cota}) is true if and only if $\bar{\psi}'(0+)>-\infty$) and
\begin{equation}
\bar{u}_T=\sup_{s\leq T}\sup_{y\in\mathbb{R}^d}u_f^*(y,s)<\infty.\notag
\end{equation}
Following the same steps as in the proof of Lemma 4 in \cite{am} the finitenes of $\bar{u}_T$ can be deduced from the fact that if we assume, without loss of generality, that $f$ is bounded by $\theta\geq0$, then 
\begin{equation}
u_f^*(y,s)\leq U_{\theta}^*(s),\qquad\text{for all $y\in\mathbb{R}^d$ and $s\geq0$},\notag
\end{equation}
where $U_{\theta}^*(s)$ is the unique solution to
\begin{equation}
U_{\theta}^*(s)+\int_0^s\bar{\psi}^*(U_{\theta}^*(u))du=\theta,\notag
\end{equation}
with
\begin{equation}
\bar{\psi}^*(\lambda)=\bar{\psi}(\lambda+\lambda^*),\qquad\text{for all $\lambda\geq-\lambda^*$}.\notag
\end{equation}
This implies that $\bar{u}_T\leq \sup_{s\leq T}U_{\theta}^*(s)<\infty$ , and thus the proof is complete.
\qed

\noindent{\bf Proof of Theorem 3.} It just suffices to prove, thanks to Lemma 2, that
\begin{equation}
\mathbf{E}_{\mu\times\nu}(e^{\langle f,I_t\rangle-\langle h,Z_t\rangle})=e^{-\langle v_{f,h}(\cdot,t),\nu\rangle}\notag
\end{equation}
where $I:=I^{\mathbb{N}^*}+I^{\mathbb{P}^*}+I^{\eta,L}+I^{\eta,NL}$, and $v_{f,h}$ solves (\ref{ed4}). Putting Lemma 5 and Lemma 6  together it suffices to show that when $g_{t-s}(\cdot)=\Phi(u_f^*(\cdot,t-s))$ (where $\Phi$ is given in (\ref{els})) we have that $\exp\{-w(x,t)\}$ is the solution to (\ref{ed4}). So making the computations as in \cite{am} it is easy to see that
\begin{align}
H_{t-s}&(\cdot,-\lambda^*e^{-w(\cdot,t-s)})-\lambda^*\Phi(u_f^*(\cdot,t-s))e^{-w(\cdot,t-s)}\notag\\
&=-\lambda^*qe^{-w(\cdot,t-s)}+\gamma\lambda^*\pi(\cdot,e^{-w(\circ,t-s)})+\beta(\lambda^*)^2e^{-2w(\cdot,t-s)}+\lambda^*qe^{-w(\cdot,t-s)}\notag\\
&-\lambda^*e^{-w(\cdot,t-s)}\left((\alpha+1)+2\beta(u_f^*(\cdot,t-s)+\lambda^*)-\int_0^{\infty}(xe^{-x(u_f^*(\cdot,t-s)+\lambda^*)}- x)\Pi^L(dx)\right)\notag\\
&+\int_0^{\infty}(e^{\lambda^*ye^{-w(\cdot,t-s)}}-1-\lambda^*ye^{-w(\cdot,t-s)})e^{-(\lambda^*+u_f^*(\cdot,t-s))y}\Pi^{L}(dy)\notag\\
&+\int_0^{\infty}(e^{\lambda^*y\pi(\cdot,e^{-w(\circ,t-s)})}-1)e^{-y\pi(\cdot,\lambda^*+u_f^*(\circ,t-s))}\Pi^{NL}(dy)\notag\\
&=\phi^{L,*}(-\lambda^*e^{-w(\cdot,t-s)}+u_f^*(\cdot,t-s))+\phi^{NL,*}(\cdot,-\lambda^*e^{-w(\cdot,t-s)}+u_f^*(\cdot,t-s))\notag\\
&-(\phi^{L,*}(u_f^*(\cdot,t-s))+\phi^{NL,*}(\cdot,u_f^*(\cdot,t-s))).\notag
\end{align}

\subsection{Backbone decomposition}
Finally with all those elements we are able to prove the following theorem which is the main result of this work. We will deal with the case when we randomize the law $\mathbf{P}_{\mu\times\nu}$ for $\mu\in\mathcal{M}_C(\mathbb{R}^d)$ by replacing the deterministic 
measure $\nu$ with a Poisson random measure having intensity measure $\lambda^*\mu$. We denote the resulting law by $\mathbf{P}_{\mu}$.
\begin{theorem}
For any $\mu\in\mathcal{M}_C(\Rd)$, the process $(\Delta,\mathbf{P}_{\mu})$ is Markovian and has the same law as $(X,\mathbb{P}_{\mu})$.
\end{theorem}
\begin{proof}
The proof is guided by the calculations found in the proof of Theorem 2 of \cite{am}. We start by addressing the claim that $(\Delta,\mathbf{P}_{\mu})$ is a Markov process. Given the Markov property of the pair $(\Delta, Z)$, it suffices to show that given $\Delta _t$ the atomic measure $Z_t$ is equal in law to a Poisson random measure with intensity $\lambda^*\Delta_t(dx)$. Thanks to Campbell's formula for Poisson random
measures, this is equivalent to showing that for all $h \in bp(\mathbb{R}^d)$,
\begin{equation}
\mathbf{E}_{\mu}(e^{-\langle h,Z_t\rangle}|\Delta_t)=\exp\{-\langle \lambda^*(1-e^{-h}),\Delta_t\rangle\},\notag
\end{equation}
which in turn is equivalent to showing that for all $f,h\in bp(\mathbb{R}^d)$,
\begin{equation}\label{f1}
\mathbf{E}_{\mu}(e^{-\langle h,Z_t\rangle-\langle f,\Delta_t\rangle})=\mathbf{E}_{\mu}(e^{-\langle \lambda^*(1-e^{-h})+f,\Delta_t\rangle}).
\end{equation}
Note from (\ref{f2}) however that when we randomize $\nu$ so that it has the law of a Poisson random measure with intensity $\lambda^*\mu(dx)$, we find the identity
\begin{equation}
\mathbf{E}_{\mu}(e^{-\langle h,Z_t\rangle-\langle f,\Delta_t\rangle})=e^{-\langle u_f^*(\cdot,t)+\lambda^*(1-e^{-v_{f,h}(\cdot,t)}),\mu\rangle}.\notag
\end{equation}
Moreover, if we replace $f$ by $\lambda^*(1-e^{-h})+f$ and $h$ by 0 in (\ref{f2}) and again randomize $\nu$ so that it has the law of a Poisson random measure with intensity $\lambda^*\mu(dx)$ then we get
\begin{equation}
\mathbf{E}_{\mu}(e^{\langle \lambda^*(1-e^{-h})+f,\Delta_t\rangle})=\exp\left\{\langle u^*_{\lambda^*(1-e^{-h})+f}(\cdot,t)+\lambda^*(1-\exp\{-v_{\lambda^*(1-e^{-h})+f,0}(\cdot,t)\}),\mu\rangle\right\}.\notag
\end{equation}
These last two observations indicate that (\ref{f1}) is equivalent to showing that for all $f,h\in bp(\mathbb{R}^d)$, $x\in \mathbb{R}^d$ and $t\geq0$,
\begin{equation}\label{f3}
u_f^*(x,t)+\lambda^*(1-e^{-v_{f,h}(x,t)})=u^*_{\lambda^*(1-e^{-h})+f}(x,t)+\lambda^*(1-e^{-v_{\lambda^*(1-e^{-h})+f,0}(x,t)}).
\end{equation}
Note that both left and right hand side of the equality above are necessarily non-negative given that they are Laplace exponents of the left and right hand sides of (\ref{f1}). Making use of (\ref{ed2}), (\ref{ed3}), 
and (\ref{ed4}), it is computationally straightforward to show that both left and right hand side of (\ref{f3}) solve (\ref{ed2}) with initial condition $f+\lambda^*(1-e^{-h})$. Since (\ref{ed2}) has a unique solution with this initial condition, namely $u_{f+\lambda^*(1-e^{-h})}(x,t)$, we conclude that (\ref{f3}) holds true. The proof of the claimed Markov property is thus complete.
\par Having now established the Markov property, the proof is complete as soon as we can show that $(\Delta, \mathbf{P}_{\mu})$ has the same semi-group as $(X,\mathbb{P}_{\mu})$. However, from the previous part of the proof we have already established that when $f,h\in bp(\mathbb{R}^d)$,
\begin{equation}
\mathbf{E}_{\mu}(e^{-\langle h,Z_t\rangle-\langle f,\Delta_t\rangle})=e^{-\langle u_{\lambda^*(1-e^{-h})+f},\mu\rangle}=\mathbb{E}_{\mu}(e^{-\langle f+\lambda^*(1-e^{-h}),X_t\rangle}).\notag
\end{equation}
In particular, choosing $h=0$ we find
\begin{equation}
\mathbf{E}_{\mu}(e^{-\langle f,\Delta_t\rangle})=\mathbb{E}_{\mu}(e^{-\langle f,X_t\rangle}),\notag
\end{equation}
which is equivalent to the equality of the semigroups of $(\Delta,\mathbf{P}_{\mu})$ and $(X,\mathbb{P}_{\mu})$.
\end{proof}

\noindent{\bf Acknowledgments.} The authors want to thank the comments of the  anonymous referee which improved the presentation of the paper.

\end{document}